\documentclass{article}
\usepackage{authblk}
\usepackage[skip=10pt plus1pt, indent=20pt]{parskip}
\usepackage{indentfirst}
\usepackage[a4paper, margin=1.5in]{geometry}
\usepackage[shortlabels]{enumitem}
\usepackage{tikz}
\usepackage{import}
\usepackage{graphicx}
\usepackage{caption}
\usepackage{subcaption}
\usepackage[export]{adjustbox}
\usepackage{wrapfig}
\usepackage{float}
\usepackage{xcolor}
\usepackage{amsmath, amssymb, amsthm}
\usepackage[colorlinks=true, allcolors=blue]{hyperref}
\usepackage{cleveref}
\usepackage{aliascnt}
\usepackage{tabularx}
\usepackage{makecell}
\usepackage[mathlines, displaymath]{lineno}

\newcolumntype{Y}{>{\centering\arraybackslash}X}
\usepackage{placeins}

\newcommand{\ior}{\textnormal{int}}

\newcommand{\qandq}{\quad \text{and} \quad}

\newtheorem{theo}{Theorem}[section]

\newaliascnt{lemma}{theo}
\newtheorem{lemma}[lemma]{Lemma}
\aliascntresetthe{lemma}
\Crefname{lemma}{Lemma}{Lemmas}

\newaliascnt{claim}{theo}

\aliascntresetthe{claim}
\Crefname{claim}{Claim}{Claims}
\crefname{claim}{claim}{claims}

\newaliascnt{prop}{theo}
\newtheorem{prop}[prop]{Proposition}
\aliascntresetthe{prop}
\Crefname{prop}{Proposition}{Propositions}
\crefname{prop}{proposition}{propositions}

\newaliascnt{obs}{theo}

\aliascntresetthe{obs}
\Crefname{obs}{Observation}{Observations}
\crefname{obs}{observation}{observations}

\newaliascnt{remark}{theo}

\aliascntresetthe{remark}
\Crefname{remark}{Remark}{Remarks}
\crefname{remark}{remark}{remarks}

\newaliascnt{alg}{theo}
\newtheorem{alg}[alg]{Algorithm}
\aliascntresetthe{alg}
\Crefname{alg}{Algorithm}{Algorithms}
\crefname{alg}{algorithm}{algorithms}

\newaliascnt{notat}{theo}

\aliascntresetthe{notat}
\Crefname{notat}{Notation}{Notations}
\crefname{notat}{notation}{notations}

\theoremstyle{definition}

\newaliascnt{deff}{theo}
\newtheorem{deff}[deff]{Definition}
\aliascntresetthe{deff}
\Crefname{deff}{Definition}{Definitions}
\crefname{deff}{definition}{definitions}

\makeatletter

\newcommand{\mapsthrough}[1]{\overset{#1}{\rightsquigarrow}}

\title{An algorithm to detect and rigorously verify blenders}

\author[1]{Andy Hammerlindl\thanks{\texttt{andy.hammerlindl@monash.edu}}}
\author[1]{Natalia McAlister\thanks{\texttt{natalia.mcalister@monash.edu} (corresponding author)}}
\author[1]{Warwick Tucker\thanks{\texttt{warwick.tucker@monash.edu}}}

\affil[1]{School of Mathematics, Monash University, Victoria 3800 Australia}

\date{ }

\begin{document}

\maketitle

\begin{abstract}
    We present a characterisation of blenders based on mapping properties of certain sets of curves that can be rigorously verified by computer-assisted methods. We develop an algorithm to construct these sets of curves that requires only a rough approximation of the strong unstable direction in a prescribed region. Since our approach does not rely on precise data, such as the exact location of invariant manifolds or fixed points, it provides a systematic framework to verify blenders in explicit examples. Here, we apply this framework to rigorously verify that a family of three-dimensional H\'enon-like maps presents blenders.\\


\end{abstract}

\section{Introduction}

Uniformly hyperbolic systems are a well-understood class of chaotic dynamical systems, since their behavior is controlled by a uniform and robust structure. However, it is known that among chaotic systems, those that are uniformly hyperbolic are rare. This observation motivates the study of chaos beyond uniform hyperbolicity. While this behaviour is expected to be a common property, proving that a given system is robustly transitive and nonhyperbolic is challenging, and as a consequence, relatively few examples are known. With this motivation, we study \textit{blenders}, a geometric object introduced by Bonatti and D\'{i}az to construct examples of robustly transitive diffeomorphisms that are not uniformly hyperbolic \cite{MR1381990}.

In this paper, for an invariant transitive hyperbolic set with $\dim E^u > u$, we define a hyperbolic blender as a family of $u$-dimensional discs that robustly intersect the stable manifold of the invariant set (see \Cref{hypblenderdef}). There have been many formal definitions for blenders, common to all of these is the idea that the stable manifold of the invariant set behaves as if it had higher dimension than it actually does \cite{MR1381990, MR2105774, MR2427422, MR2931324, MR3559345}. Since the existence of blenders is an open condition, proving their existence is tractable using rigorous numerics. However, turning this conceptual advantage into a practical and rigorous verification remains challenging, as it typically requires detailed knowledge of large pieces of the stable manifold and other features of the system.

The purpose of this work is to introduce a general algorithm for rigorously verifying the existence of blenders. The method requires only a single disc that is somewhat aligned with the strong unstable direction. The algorithm iterates this disc to generate the family of discs required by the blender definition, without relying on additional information such as the exact location of the stable manifold, fixed points, or other properties. This makes it a practical tool for finding new examples of blenders, enabling the identification and rigorous verification of chaotic behaviour in systems that would be challenging to analyse using other methods. Moreover, we apply our algorithm to a well-known family of maps, proving the existence of blenders for a wider range of parameters than previously established.

To implement this algorithm, the first step is to characterise blenders in a way that does not rely on the stable manifold. In \Cref{sec:background}, we introduce a definition of a blender as a family $\Gamma$ of $u$-dimensional discs such that every disc in $\Gamma$ has a sub-disc whose image is also in $\Gamma$, and this property persists under perturbations of the map. Moreover, we prove that if these discs are contained in a certain region $\mathcal{D}$, then $\Gamma$ is a hyperbolic blender for the maximal invariant set of $\mathcal{D}$. The idea of the proof, as detailed in \Cref{sec:background}, is that this forward invariance property implies that every disc in $\Gamma$ has a point that stays in $\mathcal{D}$ after infinitely many iterations, therefore such point is in the stable manifold of $\Lambda$.

In \Cref{theo:main}, we express a family of discs as a finite union of certain subfamilies, that we call \textit{u-tubes} and give a condition for the existence of a blender. We say that a small u-tube $\mathcal{U}_i$ \textit{maps through} a big u-tube $\mathcal{V}_j$ if every disc in $\mathcal{U}_i$ has a sub-disc in the interior of $\mathcal{V}_j$.

\begin{theo} \label{theo:main}
    Let $f:M \to M$ be a diffeomorphism. Let $\Gamma$ be a non-empty family of discs that can be expressed as finite unions
    \[\Gamma = \bigcup_{i \in I}\mathcal{U}_{i} = \bigcup_{j \in J} \mathcal{V}_{j},\] 
    where $\mathcal{U}_{i}$ and $\mathcal{V}_{j}$ are small and big u-tubes respectively. Suppose that for all $i \in I$ there exists $j \in J$ such that $\mathcal{U}_{i}$ maps through $\mathcal{V}_{j}$. Then, $\Gamma$ is a blender.
\end{theo}

\begin{figure}[h]
    \centering
    \begin{subfigure}[b]{0.45\textwidth}
        \includegraphics[width=\textwidth]{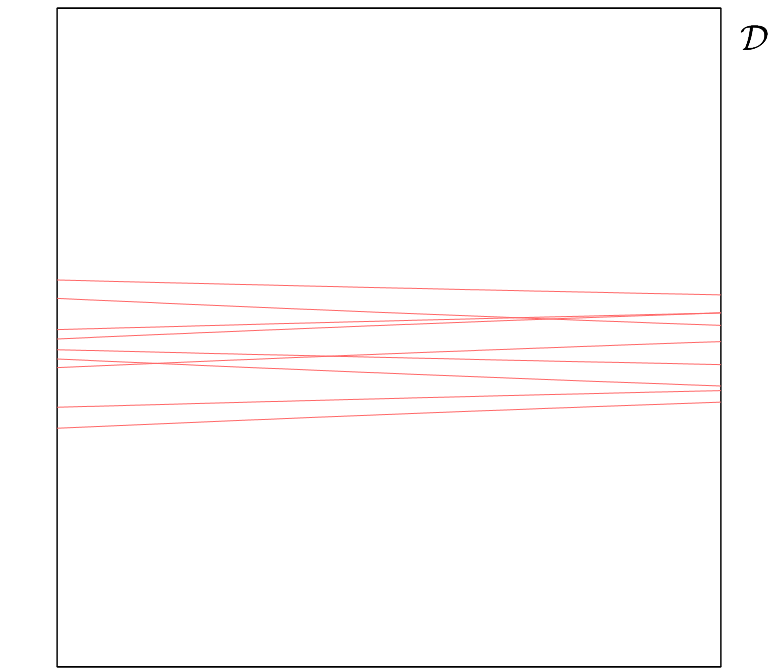}
    \end{subfigure}
    \hspace{0.75cm}
    \begin{subfigure}[b]{0.45\textwidth}
        \includegraphics[width=\textwidth]{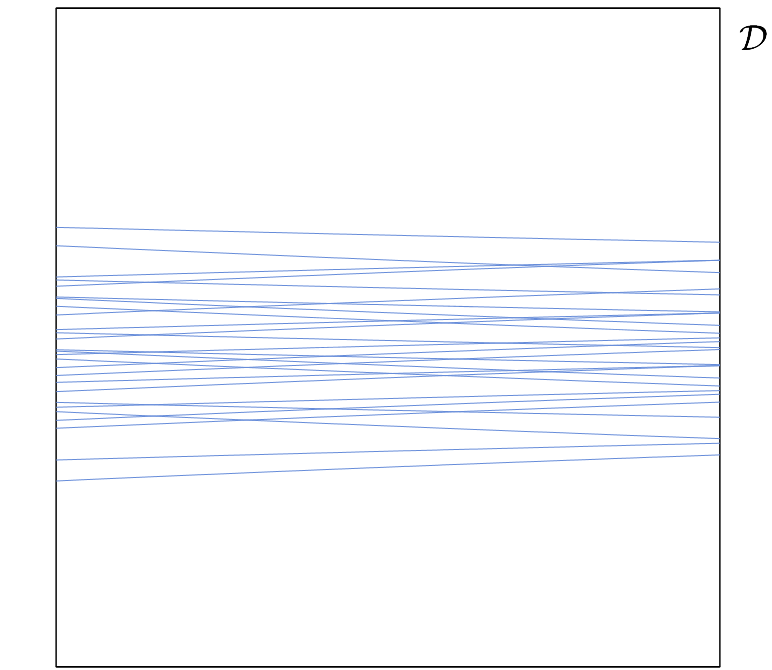}
    \end{subfigure}
    \begin{subfigure}[b]{0.49\textwidth}
        \vspace{0.5cm}
        \includegraphics[width=\textwidth]{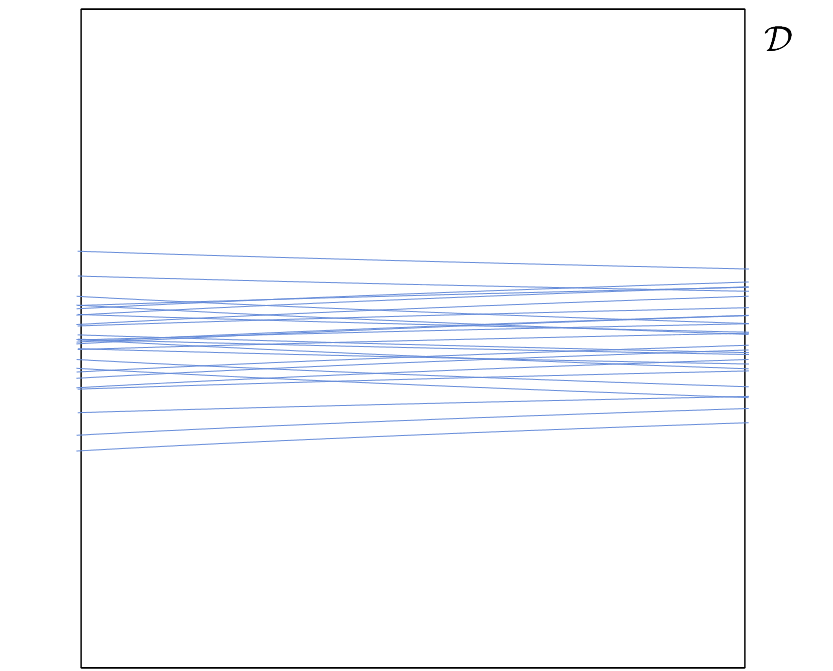}
    \end{subfigure}
    \hfill
    \begin{subfigure}[b]{0.49\textwidth}
        \vspace{0.5cm}
        \includegraphics[width=\textwidth]{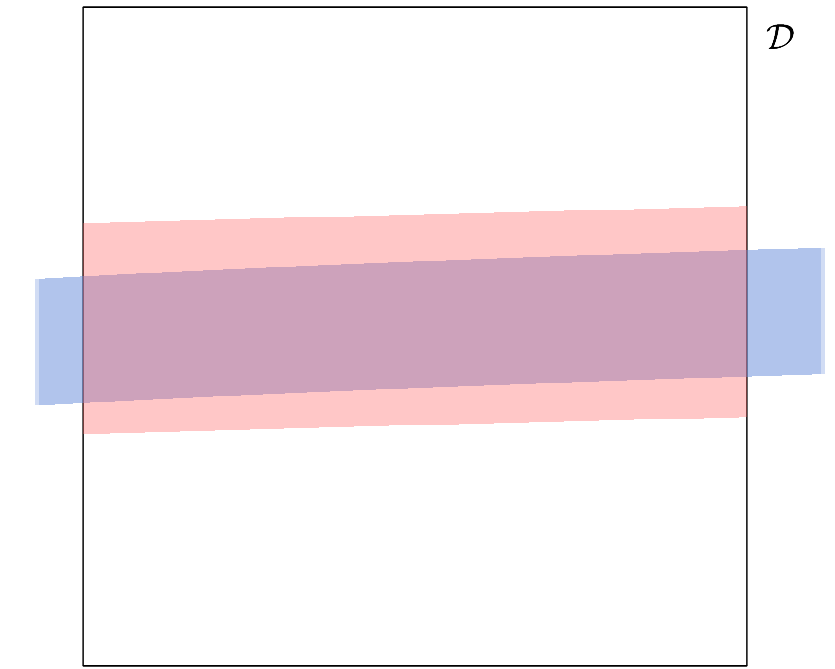}
    \end{subfigure}
    \caption{Projection onto the $xz$-plane of the objects appearing in the hypotheses of \Cref{theo:main}, for $f$ as in \eqref{eq:H} with $\xi = 1.2$. Here, the u-tubes are defined as neighbourhoods of certain curves.
    Top left: The curves defining the big u-tubes in $\Gamma$ (red).
    Top right: The curves defining the small u-tubes in $\Gamma$ (blue).
    Bottom left: Images of selected subcurves of the curves defining big u-tubes in $\Gamma$ (blue). 
    Bottom right: A subset of the image of a small u-tube $\mathcal{U}_i$ (blue), together with the big u-tube $\mathcal{V}_j$ that $\mathcal{U}_i$ maps through (red).}
    \label{fig:all_curves}
\end{figure}

In \Cref{fig:all_curves}, we illustrate a family of discs of dimension 1 in the hypothesis of \Cref{theo:main}. This theorem discretises the family of curves from the blender definition. Now, in order to verify that a map $f$ presents a blender it suffices to provide a finite set of big u-tubes, that can be expressed as a finite union of small u-tubes, and verify that any small u-tube in this family maps through one of the big u-tubes.

With this verification criterion established, we now outline the algorithm used to construct a set of u-tubes in the hypothesis of \Cref{theo:main}. We start with a big u-tube that is somewhat aligned with the strong unstable direction. Initially, the family consists of this single big u-tube, which is then split into finitely many small u-tubes. For each of these small u-tubes, we iterate it forward. If its image does not map through any of the big u-tubes already in the family, we construct a new big u-tube, which the small u-tube maps through, and add it to the family. We repeat this procedure for every big u-tube in the family. At the beginning, this process typically requires constructing many new u-tubes. However, after some iterations, we may observe that the images of the small u-tubes map through big u-tubes already in the family. If the algorithm terminates—meaning that no new u-tubes need to be added—then the verification criterion is satisfied, and the existence of a blender has been rigorously established. 

Although running this algorithm to termination is sufficient to verify the hypothesis of \Cref{theo:main}, once the family is constructed the verification can be performed much more efficiently. During the construction, we store, for each small u-tube, an identifier of the big u-tube that it maps through. After all this information is collected, we simply check that each small u-tube maps through its recorded big u-tube in the family. This approach significantly reduces runtime, as it avoids constructing new tubes or checking each small u-tube against all big u-tubes individually. Both of these algorithms are described in detail in \Cref{sec:constru_gen}.

To computationally verify that a small u-tube maps through a big u-tube we use validated numerics, a framework for computation that accounts for rounding and discretisation errors to produce rigorous results. Validated numerics rely on interval analysis, which is similar to numerical analysis but replacing real numbers with closed intervals. In this setting, real-valued functions can be extended to interval-valued ones where the output interval is guaranteed to contain the range of the real-valued function over the input interval. These methods were first introduced in \cite{MR231516} (see also \cite{MR2482682, MR2185786, MR2807595}), and they have since been used to prove numerous significant results in dynamical systems and related areas, see for instance \cite{MR1369639, MR1870856, MR3709329, MR4235311}. In our setting, by performing a change of coordinates we can represent u-tubes as cartesian products of intervals equipped with cone fields. Then, showing that a small u-tube maps through a big u-tube can be reduced to verifying a finite number of interval inclusions. This method is detailed in \Cref{sec:verif_spec}.

As an application of our method, we show the existence of blenders for a family of maps $f: \mathbb{R}^3 \to \mathbb{R}^3$ such that
\begin{equation} \label{eq:H}
    f(x, y, z) = (ax^2 + by + c, x, \xi z + x).
\end{equation}
This family was chosen since it is a conjugate to the map \begin{equation} \label{eq:F}
    F(x, y, z) = (y, \mu + y^2 + \beta x, \xi z +y),
\end{equation} with parameters related by $\mu = a c$, $b = \beta$ via the linear conjugacy $\phi(x, y, z) = (ay, ax, az)$. The family given by \eqref{eq:F} has been studied in \cite{MR3168258, MR3846435, MR4160092, MR4808810, CKO2025}. These maps present a full horseshoe dynamic on the first two coordinates for some choices of $a$, $b$ and $c$. On the third coordinate, they present expansion or contraction depending on the value of $\xi$. This family provides an explicit formula for maps with rich dynamical properties that were originally constructed in a more abstract setting. For this reason, it has been used to test different computer-assisted techniques. In particular, existence of blenders has been established for certain values of the parameter $\xi$ \cite{MR3168258, MR3846435, MR4808810}. In addition to this, there are numerical studies suggesting that blenders exist for a larger range of parameters \cite{MR3846435, MR4160092, MR4755600, CKO2025}. Our characterisation of blenders makes it possible to rigorously verify their existence for a wider set of parameter values, as stated in the following theorem.

\begin{theo} \label{theo:henon}
    Let $f$ be as in \eqref{eq:H} with $a = 4$, $b = 0.3$, $c = -2.375$, and $\xi \in (1, 1.712]$. Let $\mathcal{D} = [-1, 1] \times [-1, 1] \times [-\frac{1}{\xi -1}, \frac{1}{\xi -1}]$, and let $\Lambda = \bigcap_{n \in \mathbb{Z}} f^n(\mathcal{D})$. Then, there exists a hyperbolic blender for $f$ and $\Lambda$.
\end{theo}

\begin{figure}[h]
    \centering
    
    \begin{subfigure}[b]{0.55\textwidth} 
        \centering
        \includegraphics[width=\textwidth]{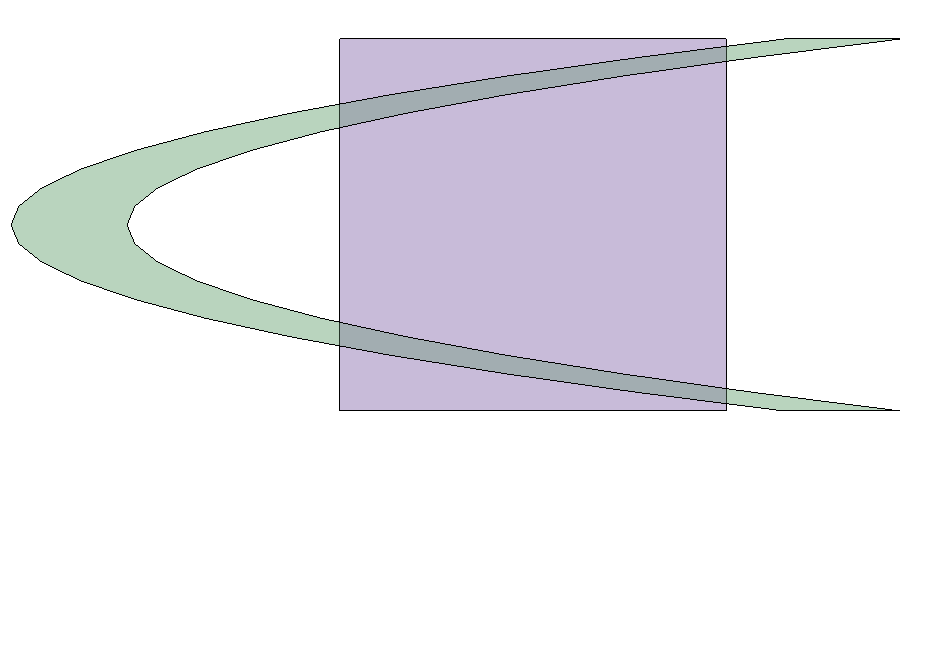}
    \end{subfigure}
    \hfill
    \begin{subfigure}[b]{0.43\textwidth}
        \centering
        \includegraphics[width=0.7\textwidth]{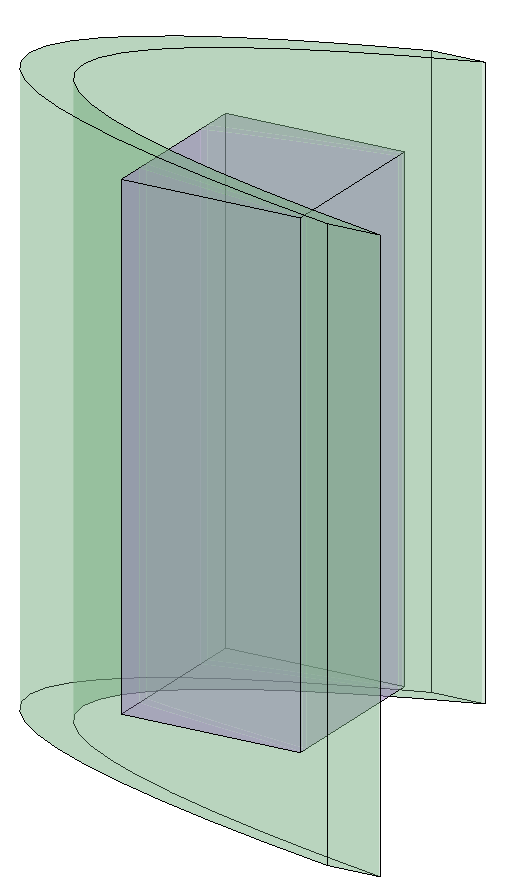}
    \end{subfigure}
    
    \caption{In purple the box $\mathcal{D}$. In green its image $f(\mathcal{D})$ projected to the $xy$-plane (left) and in three dimensions (right).}
\end{figure}

We note that since the existence of blenders is a $C^1$-robust property, \Cref{theo:henon} implies that the family of maps given by \eqref{eq:H} exhibits blenders for every $(a, b, c)$ in an open neighbourhood of $(4, 0.3, -2.375)$ and $\xi \in (1, 1.712 + \epsilon)$, for some $\epsilon > 0$. We work with the system given by \eqref{eq:H} rather than the one given by \eqref{eq:F} for computational convenience. While \eqref{eq:F} exhibits a full horseshoe dynamic in the first two coordinates over the domain $[-4, 4] \times [-4, 4]$, the conjugacy to \eqref{eq:H} with $a = 4$ enables us to obtain the same horseshoe dynamic over the domain $[-1, 1] \times [-1, 1]$. This allows us to represent curves passing through this region by polynomials defined on the interval $[-1,1]$ which simplifies certain computations. This choice is reflected in the first two coordinates of $\mathcal{D}$ in \Cref{theo:henon}. The third coordinate of $\mathcal{D}$ is chosen so that $\Lambda$ contains all bounded orbits of $f$. 

We prove \Cref{theo:henon} by applying the algorithm described above. Recall that to construct the family of curves, all that is required is a curve roughly aligned with the strong unstable direction. In this example, it is sufficient to take $\beta(t) = (t, y_0, 0)$, for some constant $y_0$, showing that precise alignment is not required.

The paper is organised as follows. In \Cref{sec:background}, we introduce definitions and prove \Cref{theo:main}. In \Cref{sec:constru_gen}, we describe two algorithms: one to construct the family of curves in the hypothesis of \Cref{theo:main}, and another to verify that a given family satisfies this hypothesis. In \Cref{sec:specdefs} we define u-tubes specific to the H\'enon-like family of maps given by \eqref{eq:H}. In \Cref{sec:verif_spec}, we present a method to verify that a given small u-tube maps through a given big u-tube, a step used throughout the construction and verification algorithms. \Cref{sec:verif_spec} and \Cref{sec:constru_spec} provide specifications to implement these algorithms for the H\'enon-like family of maps given by \eqref{eq:H}. In \Cref{sec:results}, we present the results obtained from running the program. Finally, \Cref{sec:discussion} discusses the results, some limitations of our method, and directions for future work. All computations are implemented using the CAPD library \cite{MR4283203} and the code is attached as supplementary material.

\textbf{Acknowledgements.} This research was partially supported by the Australian Research Council through grant DP220100492. The authors would like to thank Marisa Cantarino, Maciej Capi\'nski, Dana C'Julio, Bernd Krauskopf and Hinke Osinga for helpful conversations and valuable insights. The authors acknowledge the use of Claude AI (Anthropic) for proofreading and language polishing.

\section{Background, definitions and proof of \Cref{theo:main}} \label{sec:background}

We begin this section by introducing the definition of a blender used throughout this paper and comparing it with standard definitions in the literature. We then formally introduce the concepts required to state and prove \Cref{theo:main}, which provides a characterisation of blenders that can be verified computationally. In what follows, $M$ denotes a manifold without boundary, $d$ denotes its dimension $d = \dim M \geq 3$ and $u$ is an integer satisfying $1 \leq u < d$. A disc is the image of an element of $C^1(\mathbb{D}^u, M)$, where $\mathbb{D}^u$ is the closed unit disc in $\mathbb{R}^u$. 

\begin{deff} [Blender] \label{blenderdef}
    Let $f:M \to M$ be a diffeomorphism. A \textit{cu-blender of index $u$ for $f$} is a family $\Gamma$ of discs of dimension $u$ with the following property: there is a $C^1$-open neighbourhood $\mathcal{F}$ of $f$ such that for every $g \in \mathcal{F}$ and every $\gamma \in \Gamma$, there is a subset $\sigma \subset \gamma$ such that $g (\sigma) \in \Gamma$.
\end{deff}

We note that there are many definitions of blenders in the literature, most of them involve robust intersections with the stable manifold of a hyperbolic set \cite{MR2105774, MR2931324, MR3508160}. In \Cref{prop:blenderdefs}, we show that our blender captures this property for a suitable hyperbolic set. We begin by introducing the notions of hyperbolic set and stable and unstable manifolds.

\begin{deff}[Hyperbolic set] Given a map $f: M \to M$, we say that an invariant compact set $\Lambda \subset M$ is \textit{hyperbolic} if the tangent bundle over $\Lambda$ admits a continuous decomposition \[T_{\Lambda}M = E^u \oplus E^s\] invariant under $Df$, and such that there exists $\lambda < 1$ such that for all $p \in \Lambda$, \[||Df^{-1}(p) v_u|| < \lambda ||v_u||,\] \[||Df(p) v_s|| < \lambda ||v_s||,\] for all $v_u \in E_p^u$, $v_s \in E_p^s$, and a choice of a Riemannian metric on $M$.
\end{deff}

\begin{deff}[Stable and unstable manifolds] For a given point $p$ in a hyperbolic set $\Lambda$ we define the \textit{stable and unstable manifolds of $p$} as \[
W^s(p) = \{q \in M : \lim_{k\to \infty} d(f^k(p), f^k(q)) = 0 \},
\]
\[
W^u(p) = \{q \in M : \lim_{k\to \infty} d(f^{-k}(p), f^{-k}(q)) = 0 \}.
\]
We define the \textit{stable and unstable manifolds of $\Lambda$} as \[W^s(\Lambda) = \underset{p \in \Lambda}{\bigcup} W^s(p) \qandq W^u(\Lambda) = \underset{p \in \Lambda}{\bigcup} W^u(p).\]
For a neighbourhood $U$ of $p$, let $C_p(A)$ denote the connected component of a set $A$ containing $p$. We define the \textit{local stable and unstable manifolds of $p$} as
\[
W^s_{\textnormal{loc}}(p) = C_p(W^s(p) \cap U) \qandq W^u_{\textnormal{loc}}(p) = C_p(W^u(p) \cap U).
\]
We define the \textit{local stable and unstable manifolds of $\Lambda$} as \[W^s_{\textnormal{loc}}(\Lambda) = \underset{p \in \Lambda}{\bigcup} W^s_{\textnormal{loc}}(p) \qandq W^u_{\textnormal{loc}}(\Lambda) = \underset{p \in \Lambda}{\bigcup} W^u_{\textnormal{loc}}(p). \] 
\end{deff}

In our setting, the hyperbolic set $\Lambda$ is defined as the maximal invariant subset of a set $\mathcal{D}$. Throughout, we take $\mathcal{D}$ as the neighbourhood $U$ used in the definition of the local stable and unstable manifolds, although we omit this specification for brevity.

We now introduce the notion of \textit{hyperbolic blender}, following \cite{MR2931324}.

\begin{deff}[Hyperbolic blender] \label{hypblenderdef} Let $f: M \to M$ be a diffeomorphism and $\Lambda \subset M$ an invariant transitive hyperbolic set with $\dim E^u > u$. A \textit{hyperbolic cu-blender of index $u$ for $f$ and $\Lambda$} is a non-empty family $\Gamma$ of discs of dimension $u$ with the following property: there exists a $C^1$-open neighbourhood $\mathcal{F}$ of $f$ such that for every $g \in \mathcal{F}$ and every $\gamma \in \Gamma$, \[\gamma \cap W^s_{\textnormal{loc}}(\Lambda_g) \neq \emptyset,\]
where $\Lambda_g$ is the continuation of $\Lambda$ for $g$.
\end{deff}

We note that in several definitions found in the literature, the family of discs $\Gamma$ is explicitly required to be $C^1$-open. However, this requirement is redundant, as openness follows from the $C^1$-robustness over the neighbourhood $\mathcal{F}$. To see this, consider a $C^1$-open neighbourhood $\mathcal{G}$ of $f$ and a $C^1$-open neighbourhood $\mathcal{H}$ of the identity map on $M$ such that $h^{-1} \circ g \circ h \in \mathcal{F}$ for all $g \in \mathcal{G}$ and $h \in \mathcal{H}$. We define a $C^1$-neighbourhood of $\Gamma$ by $\mathcal{N} := \{ h (\gamma) : h \in \mathcal{H}, \gamma \in \Gamma \}$. For any disc $\Tilde{\gamma} = h (\gamma) \in \mathcal{N}$ and any $g \in \mathcal{G}$, we have $h^{-1} \circ g \circ h \in \mathcal{F}$, which implies that $\gamma$ intersects $W^s_{\text{loc}}(\Lambda_{h^{-1} \circ g \circ h})$. Since $h$ conjugates $h^{-1} \circ g \circ h$ to $g$, it follows that $W^s_{\text{loc}}(\Lambda_{h^{-1} \circ g \circ h}) = h^{-1}(W^s_{\text{loc}}(\Lambda_g))$. Thus, $\Tilde{\gamma}$ intersects $W^s_{\text{loc}}(\Lambda_g)$. Beyond simplifying the definition, omitting this requirement is particularly convenient for our computational framework. The following proposition relates blenders (\Cref{blenderdef}) to hyperbolic blenders (\Cref{hypblenderdef}).

\begin{prop} \label{prop:blenderdefs}
    Let $f:M \to M$ be a diffeomorphism and let $\Gamma$ be a cu-blender of index $u$ for $f$. Let $\mathcal{D} \subset M$ be a compact set such that $\Lambda = \bigcap_{n \in \mathbb{Z}} f^n(\mathcal{D})$ is a transitive hyperbolic set with $\dim E^u > u$. If every disc $\gamma \in \Gamma$ is contained in $\mathcal{D}$, then $\Gamma$ is a hyperbolic cu-blender of index $u$ for $f$ and $\Lambda$.
\end{prop}

\begin{proof}
    Let $\mathcal{F}$ be the $C^1$-open neighbourhood of $f$ given by \Cref{blenderdef}. By shrinking $\mathcal{F}$ if necessary, we may assume that the continuation $\Lambda_g$ is well-defined and hyperbolic for all $g \in \mathcal{F}$. Fix $g \in \mathcal{F}$ and $\gamma \in \Gamma$. To show that $\gamma \cap W^s_{\text{loc}}(\Lambda_g) \neq \emptyset$, we begin by setting $\gamma_0 := \gamma$ and inductively construct a sequence of discs in $\Gamma$. By definition of $\Gamma$, for each $n \geq 0$, there exists a subset $\sigma_n \subset \gamma_n$ such that $\gamma_{n+1} := g ( \sigma_n) \in \Gamma$. Note that $g^{-1}(\gamma_{n+1}) \subset \gamma_n$. By induction, this yields the nested inclusion $g^{-n}(\gamma_n) \subset g^{-(n-1)}(\gamma_{n-1}) \subset \dots \subset \gamma$. This forms a nested sequence of non-empty compact subsets of $\gamma$. Consequently, there exists a point $p \in \bigcap_{n > 0} g^{-n}(\gamma_n) \subset \gamma$. By construction, $g^n(p) \in \gamma_n$ for all $n \geq 0$. We now show that $p \in W^s_{\text{loc}}(\Lambda_g)$.

    For all $n \in \mathbb{N}$, let $x_n := g^n(p)$. Since $\mathcal{D}$ is compact, any convergent subsequence of $\{ x_n \}$ converges to a point in $\mathcal{D}$. For any such subsequence $\{ x_{n_k} \}$, denote its limit by $y \in \mathcal{D}$. Note that $g^j(x_{n_k}) \in \mathcal{D}$ for all $j > -n_{k}$. Then, for a fixed $j \in \mathbb{Z}$, we have that $g^j(x_{n_k}) \in \mathcal{D}$ for all large $k$. Therefore, the limit $g^j(y)$ is also in the closed set $\mathcal{D}$, showing that $y \in \Lambda_g$. Since this holds for any converging subsequence $\{ x_{n_k} \}$, we have $\textnormal{dist} (x_n, \Lambda_g) \to 0$. Then, by a shadowing argument, it follows that $p \in W^s_{\text{loc}}(q)$, for some $q \in \Lambda_g$.
\end{proof}

In our computer implementation, the blender is represented by a finite union of subfamilies of a specific form, each determined by a finite amount of data. We call these subfamilies \textit{u-tubes}. In practice, u-tubes will be $C^1$-neighbourhoods around discs roughly aligned with the strong unstable directions. This discretisation allows us to verify the blender property in a finite number of steps by verifying that each u-tube maps through another u-tube in the family. However, since our maps are centre-expanding, they stretch u-tubes in the centre directions, making it difficult to verify this property. For this reason, we consider \textit{small u-tubes} and \textit{big u-tubes} such that a big u-tube can be expressed as a finite union of small u-tubes, and a small u-tube can map through a big u-tube. In practice, this splitting is done along the centre direction. The precise definitions are as follows:

\begin{deff} \label{deff:splitsinto}
    We say that a big u-tube $\mathcal{V}$ \textit{splits into} small u-tubes $\mathcal{U}_{1}, \dots, \mathcal{U}_{n}$ if their union satisfies $\mathcal{V} = \mathcal{U}_{1} \cup \dots \cup \mathcal{U}_{n}$. We refer to $\{ \mathcal{U}_i \}_{i=1}^n$ as a \textit{splitting} of $\mathcal{V}$.
\end{deff}

\begin{deff} \label{deff:mapsthrough}
    We say that a small u-tube $\mathcal{U}$ \textit{maps through} a big u-tube $\mathcal{V}$, denoted $\mathcal{U} \mapsthrough{f} \mathcal{V}$, if there exists a $C^1$-open neighbourhood $\mathcal{F}$ of $f$ such that for every $\gamma \in \mathcal{U}$, there exists $\sigma \subset \gamma$ such that $g ( \sigma) \in \mathcal{V}$ for every $g \in \mathcal{F}$. 
\end{deff}

With all the relevant notions in place, we can now proceed to the proof of \Cref{theo:main}.

\begin{proof}[Proof of \Cref{theo:main}] For each $i \in I$, there exists $j \in J$ such that $\mathcal{U}_i \mapsthrough{f} \mathcal{V}_j$. That is, there exists a $C^1$-open neighbourhood $\mathcal{F}_i$ of $f$ such that for each $\gamma \in \mathcal{U}_i$, there exists $\sigma \subset \gamma$ such that $g ( \sigma ) \in \mathcal{V}_j$ for every $g \in \mathcal{F}_i$. Since $I$ is finite the intersection $\mathcal{F} = \bigcap_{i \in I} \mathcal{F}_i$ is a $C^1$-open neighbourhood satisfying that for every $\gamma \in \Gamma$, there exists $\sigma \subset \gamma$ such that $g ( \sigma) \in \Gamma$ for every $g \in \mathcal{F}$.
 
\end{proof}

\section{General algorithms} \label{sec:constru_gen}

In this section we present an algorithm to build two finite sets of u-tubes $\mathcal{K} := \{\mathcal{U}_i \}_{i \in I}$ and $\mathcal{L} := \{\mathcal{V}_j \}_{j \in J}$, satisfying the hypotheses of \Cref{theo:main}. Here we present the most general version of the algorithm. In what follows we assume that we have sub-algorithms to: (i) split a big u-tube into appropriate small u-tubes, (ii) verify that a small u-tube maps through a big u-tube, and (iii) construct, given a small u-tube $\mathcal{U}$, a big u-tube $\mathcal{V}$ such that $\mathcal{U} \mapsthrough{f} \mathcal{V}$. The specific details of these sub-algorithms for the H\'enon-like family of maps given by \Cref{eq:H} are presented in Sections \ref{sec:specdefs}, \ref{sec:verif_spec} and \ref{sec:constru_spec}, respectively.

The overall structure of the algorithm that builds the sets of u-tubes $\mathcal{K}$ and $\mathcal{L}$ is as follows. We first consider an appropriate initial big u-tube somewhat aligned with the unstable directions, say $\mathcal{V}_1$. We initialise $\mathcal{L} = \{\mathcal{V}_1\}$ and add big u-tubes to this set until a certain condition is satisfied. 
As we construct $\mathcal{L}$, the set $\mathcal{K}$ will be determined as the union of splittings of all elements of $\mathcal{L}$. For every $\mathcal{U} \in \mathcal{K}$ we give $\mathcal{V} \in \mathcal{L}$ such that $\mathcal{U} \mapsthrough{f} \mathcal{V}$. We also consider a set of big u-tubes $\mathcal{P}$, which consists of the big u-tubes in $\mathcal{L}$ pending to be processed. We initialise $\mathcal{P} = \mathcal{L}$ and add and remove elements from $\mathcal{P}$. Throughout the algorithm we have the following invariant condition: \begin{itemize}
    \item [(A)] For any $\mathcal{U} \in \mathcal{K}$ that arises from splitting a big u-tube in $\mathcal{L} \setminus \mathcal{P}$, there is $\mathcal{V} \in \mathcal{L}$ satisfying $\mathcal{U} \mapsthrough{f} \mathcal{V}$. \label{cond_A}
\end{itemize} Then, if $\mathcal{P}$ becomes empty we have obtained the required sets $\mathcal{K}$ and $\mathcal{L}$. Using this notation, the step-by-step execution is outlined in \Cref{alg:main}.

\begin{alg} \label{alg:main}
\end{alg}
\begin{enumerate}
    \item Initialise $\mathcal{P} = \mathcal{L} = \{\mathcal{V}_1\}$.
    \item Take an arbitrary element $\hat{\mathcal{V}} \in \mathcal{P}$ and consider a splitting $\{ \mathcal{U}_i \}_{i=1}^n$. For each $i \in \{1, ..., n\}$,
        \begin{enumerate}
            \item Let $\mathcal{U} = \mathcal{U}_i$. Check if there exists $\mathcal{V} \in \mathcal{L}$ such that $\mathcal{U} \mapsthrough{f} \mathcal{V}$. If not,
            \begin{enumerate}
                \item Construct a big u-tube $\mathcal{V}$ satisfying $\mathcal{U} \mapsthrough{f} \mathcal{V}$.
                \item Add $\mathcal{V}$ to $\mathcal{L}$ and $\mathcal{P}$.
            \end{enumerate}
        \end{enumerate}
    \item Remove $\hat{\mathcal{V}}$ from $\mathcal{P}$.
    \item If $\mathcal{P}$ is not empty, go to step 2.
\end{enumerate}

Since $\mathcal{L} \setminus \mathcal{P}$ is the empty set at the beginning of the algorithm, it is clear that (A) holds at this point. Every time we apply step 2, every small u-tube $\mathcal{U} \in \mathcal{K}$ arising from a splitting of $\hat{\mathcal{V}}$ is assigned a big u-tube $\mathcal{V} \in \mathcal{L}$ such that $\mathcal{U} \mapsthrough{f} \mathcal{V}$, regardless of whether $\mathcal{V}$ was constructed at this step or was already in $\mathcal{L}$. Then, after step 2 is completed we may remove $\hat{\mathcal{V}}$ from $\mathcal{P}$ and still have (A), which demonstrates that it is indeed an invariant condition. Note that at step 2 we add between 0 and $n$ new u-tubes to $\mathcal{P}$, so there is no guarantee that it will become empty and the algorithm will end. But it turns out that for the particular family of maps we are considering and for the specifications given later on in \Cref{sec:results}, it does eventually terminate, as shown in \Cref{fig:noElements}.

\begin{figure}[h]
    \centering
    \includegraphics[width=0.95\linewidth]{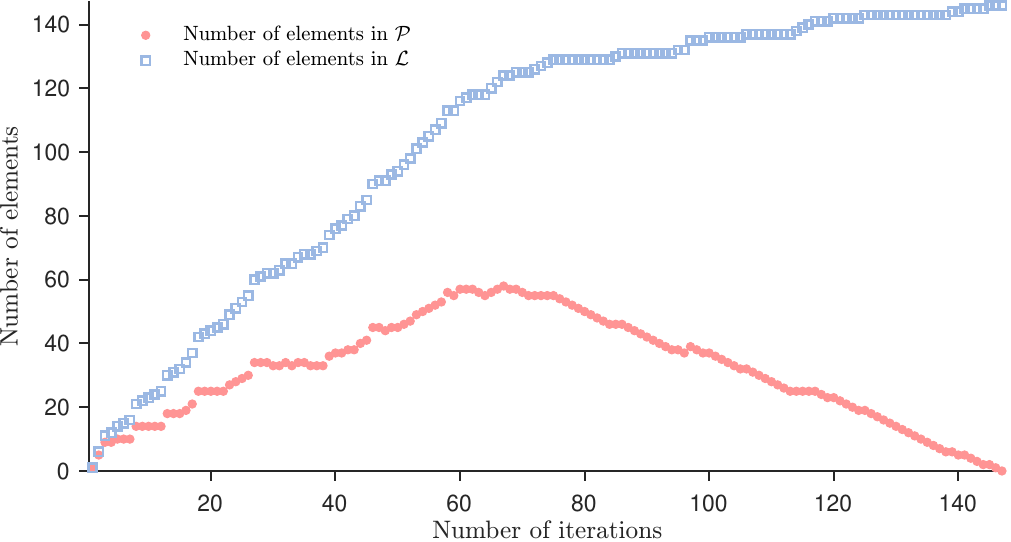}
    \caption{Number of elements in $\mathcal{L}$ and $\mathcal{P}$ at each iteration of \Cref{alg:main}, for $f$ as in \Cref{theo:henon} with $\xi = 1.5$.}
    \label{fig:noElements}
\end{figure}

Once we have constructed the sets of u-tubes $\mathcal{K}$ and $\mathcal{L}$, we can store them, reducing the computational effort to verifying that, for each $\mathcal{U} \in \mathcal{K}$, there is $\mathcal{V} \in \mathcal{L}$ such that $\mathcal{U} \mapsthrough{f} \mathcal{V}$. This approach significantly reduces runtime. In practice we store an ordered list of u-tubes representing $\mathcal{L}$ and an ordered list of mapping data, which we denote $\mathcal{M}$. The verification proceeds as follows:

\begin{alg} \label{alg:verif}
\end{alg}
\begin{enumerate}
    \item Split the big u-tubes in $\mathcal{L}$ to generate an ordered list of small u-tubes representing $\mathcal{K}$ that preserves the order inherited from $\mathcal{L}$.
    \item For each $\mathcal{U} \in \mathcal{K}$, the ordered list $\mathcal{M}$ provides a big u-tube $\mathcal{V} \in \mathcal{L}$. Verify that $\mathcal{U} \mapsthrough{f} \mathcal{V}$.
\end{enumerate}

\section{U-tubes for the H\'enon-like system} \label{sec:specdefs}

Having described the general method to verify the existence of blenders, we now apply it to prove \Cref{theo:henon}. More precisely we consider the H\'enon-like family of maps $f: \mathbb{R}^3 \to \mathbb{R}^3$ given by
\begin{equation*}
    f(x, y, z) = (ax^2 + by + c, x, \xi z + x),
\end{equation*} and show that it presents hyperbolic blenders for $a = 4$, $b = 0.3$, $c = -2.375$ and $\xi \in (1, 1.712]$ by constructing a family $\Gamma$ of curves in the box $\mathcal{D}$ that satisfies the hypothesis of \Cref{theo:main} via \Cref{alg:main}. We use standard methods to show that $\Lambda$ is transitive and hyperbolic with unstable dimension 2 (see \cite{MR2736320, MR4808810}).

In this section, we define the u-tubes needed for \Cref{theo:main}. To this end, we introduce curves that can be represented computationally. A \textit{u-curve} is a curve $\alpha:I \to \mathcal{D}$ such that \[\alpha(t) = (t, p_y(t), p_z(t)),\] where $p_y$ and $p_z$ are polynomials. Note that u-curves are somewhat aligned with the strong unstable direction. By a slight abuse of notation, we will identify a curve $\gamma$ with its image $\gamma(I) \subset \mathcal{D}$. Let $\epsilon_y$, $\hat{\epsilon}_y$, $\epsilon_z$ and $\hat{\epsilon}_z$ be positive real numbers. We define a u-tube around $\alpha$ as \begin{multline*}
    \mathcal{U}(\alpha, \epsilon_y, \hat{\epsilon}_y, \epsilon_z, \hat{\epsilon}_z) := \{ \gamma : I \to \mathcal{D} \text{, } C^1 \text{-curve with } \gamma(t) = (t, \gamma_y(t), \gamma_z(t)) \text{ such that }\\ 
    |\gamma_y - p_y| \leq \epsilon_y, |\gamma_y' - p_y'| \leq \hat{\epsilon}_y, |\gamma_z - p_z| \leq \epsilon_z, |\gamma_z' - p_z'| \leq \hat{\epsilon}_z \}.
    \end{multline*}

Even though we cannot explicitly represent every curve in a u-tube, we will see that we can use interval arithmetic to verify the specific properties we are interested in. In the following lemma we define big and small u-tubes such that a big u-tube splits into small u-tubes. Here, the positive integer $n$ represents the number of small u-tubes whose union forms a big u-tube, and the positive integer $k$ represents the proportion between the sizes of the big and small u-tubes along the $z$-axis. Additionally, $|I|$ denotes the length of the interval $I$. \Cref{fig:split_tube} illustrates such a decomposition for $n = 3$ and $k = 2$, and in \Cref{sec:results} we discuss how different values of $n$ and $k$ affect the outcome.

\begin{figure}[h] 
  \centering
  \def\svgwidth{1\linewidth}
  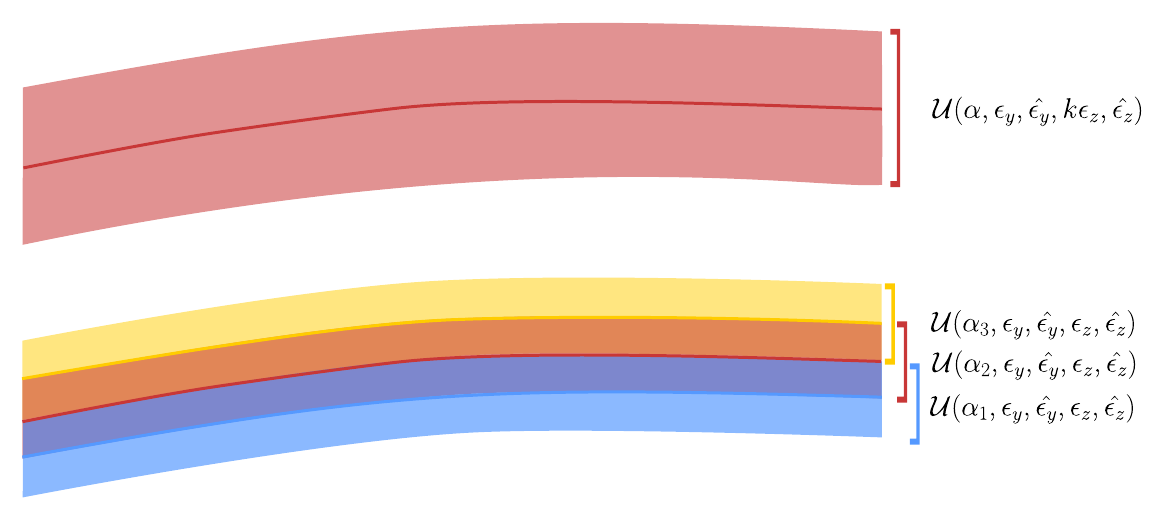
  \caption{The projection to the $xz$-plane of the u-tubes from \Cref{lemma:split} for $n = 3$, $k = 2$. Above is the big u-tube around $\alpha$. Below is the same set represented as the union of the small u-tubes around $\alpha_1$ (in blue), $\alpha_2$ (in red), and $\alpha_3$ (in yellow).} \label{fig:split_tube}
\end{figure}
\newpage
\begin{lemma} \label{lemma:split}
    Let $\alpha:I \to \mathcal{D}$ be a u-curve such that $\alpha(t) = (t, p_y(t), p_z(t))$. Let $n$, $k$ be natural numbers with $n>k>1$. Then, there exist u-curves $\alpha_i:I \to \mathcal{D}$, for $i \in \{1,..., n\}$, and a constant $\delta(k, n) = \frac{2(n-k)}{(n-1)|I|}$ such that
    \[\mathcal{U}(\alpha, \epsilon_y, \hat{\epsilon}_y, k \epsilon_z, \hat{\epsilon}_z) = \bigcup_{i=1}^{n} \mathcal{U}(\alpha_i, \epsilon_y, \hat{\epsilon}_y, \epsilon_z,  \hat{\epsilon}_z),\] 
    for any positive real numbers $\epsilon_y$, $\hat{\epsilon}_y$, $\epsilon_z$, and $\hat{\epsilon}_z = \delta(k, n) \epsilon_z$. We refer to $\mathcal{U}(\alpha, \epsilon_y, \hat{\epsilon}_y, k \epsilon_z, \hat{\epsilon}_z)$ and $\mathcal{U}(\alpha_i, \epsilon_y, \hat{\epsilon}_y, \epsilon_z, \hat{\epsilon}_z)$ as the big and small u-tubes around $\alpha$ and $\alpha_i$ respectively.
\end{lemma}

\begin{proof}
We define the u-curves $\alpha_i$ as translations of $\alpha$ in $z$ by constants uniformly spaced in $[-\epsilon_z(k-1) , \epsilon_z(k-1) ]$. Explicitly, \[ \alpha_i(t) = \alpha(t) + \left(0, 0, \left(-(k-1) + i\,\frac{2(k-1)}{n-1}\right)\epsilon_z\right), \] for $i \in \{1, \cdots, n \}$. It is clear that the union of the small u-tubes around the u-curves $\alpha_i$ is included in the big u-tube around $\alpha$. For the other inclusion, let $\gamma \in \mathcal{U}(\alpha, \epsilon_y, \hat{\epsilon}_y, k\epsilon_z, \hat{\epsilon}_z)$ such that $\gamma(t) = (t, \gamma_y(t), \gamma_z(t))$. Since $|(\gamma_z - p_z)'| < \hat{\epsilon}_z$ on $I$, we have that for any $t_1, t_2 \in I$ it holds
\[
|(\gamma_z - p_z)(t_2) - (\gamma_z - p_z)(t_1)| < \hat{\epsilon}_z |I|.
\]
 Moreover, consecutive small u-tubes are centred at $\alpha_i$ and $\alpha_{i+1}$, which differ in $z$ by $\frac{2(k-1)}{n-1}\epsilon_z$, and each has width $2\epsilon_z$ in $z$, so their overlap is 
\[
2\epsilon_z - \frac{2(k-1)}{n-1}\epsilon_z = \frac{2(n-k)}{n-1}\epsilon_z = \hat{\epsilon}_z |I|.
\]
Thus the image of $\gamma_z - p_z$ over $I$ is an interval of length less than the overlap in $z$ between consecutive small u-tubes, and therefore $\gamma$ must be contained in at least one of them.
\end{proof}

From now on, we fix $\epsilon_y = \hat{\epsilon}_y$, $\epsilon_z$, $n$, and $k$, and simply denote the small and big u-tubes around $\alpha$ as \[\mathcal{U}_\alpha := \mathcal{U}(\alpha, \epsilon_y, \epsilon_y, \epsilon_z, \hat{\epsilon}_z), \] \[\mathcal{V}_{\alpha} := \mathcal{U}(\alpha, \epsilon_y, \epsilon_y, k\epsilon_z, \hat{\epsilon}_z)\] respectively.

\section{Verification for the H\'enon-like system} \label{sec:verif_spec}

In both \Cref{alg:main} and \Cref{alg:verif} we need to check that a given small u-tube $\mathcal{U}_{\alpha}$ maps through a given big u-tube $\mathcal{V}_{\beta}$. In this section we present a characterisation of this condition that can be verified computationally.

Let us first introduce some notation. Let $\alpha(t) = (t, p_y(t), p_z(t))$ and $\beta(t) = (t, q_y(t), q_z(t))$. Recall that we fixed constants $k$, $\epsilon_y$, $\epsilon_z$ and $\hat{\epsilon}_z$ that define u-tubes. We write $\boldsymbol{\epsilon_y} := [-\epsilon_y, \epsilon_y]$, $\boldsymbol{\epsilon_z} := [-\epsilon_z, \epsilon_z]$, and $\boldsymbol{\hat{\epsilon}_z} := [-\hat{\epsilon}_z, \hat{\epsilon}_z]$. For any point $a \in \mathbb{R}^3$ we define the unstable cone $\mathcal{C}^u \subset T_a\mathbb{R}^3$ as the set of all vectors $v = (v_x, v_y, v_z) \in T_a\mathbb{R}^3$ such that $\frac{v_y}{v_x} \in \boldsymbol{\epsilon_y}$ and $\frac{v_z}{v_x} \in \boldsymbol{\hat{\epsilon}_z}$. The interior of the unstable cone, denoted $\ior(\mathcal{C}^u)$, consists of all vectors $v = (v_x, v_y, v_z) \in T_a\mathbb{R}^3$ such that $\frac{v_y}{v_x} \in \ior( \boldsymbol{\epsilon_y})$ and $\frac{v_z}{v_x} \in \ior( \boldsymbol{\hat{\epsilon}_z})$. For a given u-curve $\alpha$ with $\alpha(t) = (t, p_y(t), p_z(t))$ and a neighbourhood $U(I)$ of $I$ we define a diffeomorphism $\varphi_{\alpha}: U(I) \times \mathbb{R} \times \mathbb{R} \to U(I) \times \mathbb{R} \times \mathbb{R}$ such that \[\varphi_{\alpha}(x, y, z) = (x, y + p_y(x), z + p_z(x)).\] Recall that a u-tube around $\alpha$ is a collection of curves $\gamma: I \to \mathbb{R}^3$ with $\gamma(t) = (t, \gamma_y(t), \gamma_z(t))$, where $\gamma_y$ and $\gamma_z$ are $C^1$-close to $p_y$ and $p_z$ respectively. It follows that $\varphi_{\alpha}$ maps the box $I \times \boldsymbol{\epsilon_y} \times \boldsymbol{\epsilon_z}$ (or $I \times \boldsymbol{\epsilon_y} \times k\boldsymbol{\epsilon_z}$) endowed with the cone field $\mathcal{C}^u$ to the u-tube $\mathcal{U}_{\alpha}$ (or $\mathcal{V}_{\alpha}$). Given two u-curves $\alpha$ and $\beta$, we define \[f_{\alpha , \beta} = \varphi_{\beta}^{-1} \circ f \circ \varphi_{\alpha}.\] As suggested by \Cref{fig:fab}, in order to verify that $\mathcal{U}_{\alpha} \mapsthrough{f} \mathcal{V}_{\beta}$ it is enough to check a similar condition for $f_{\alpha, \beta}$, which we make precise below. 
\begin{deff} \label{deff:through_box}
    Let $\alpha$ and $\beta$ be two u-curves. Let $T \subset I$. We say that a box of the form $B = T \times \boldsymbol{\epsilon_y} \times \boldsymbol{\epsilon_z}$ is a \textit{through box} for $\alpha$ and $\beta$ if the following conditions hold:
    \begin{enumerate}
        \item The two connected components of $f_{\alpha, \beta}(\partial T \times \boldsymbol{\epsilon_y} \times \boldsymbol{\epsilon_z})$ are contained in different connected components of $\mathbb{R}^3 \setminus (I \times \mathbb{R} \times \mathbb{R})$. 
        \item $f_{\alpha, \beta}(p) \in \ior(\mathbb{R} \times \boldsymbol{\epsilon_y} \times k\boldsymbol{\epsilon_z})$ for all $p \in B$ with $f_{\alpha, \beta}(p) \in I \times \mathbb{R} \times \mathbb{R}$.
        \item $Df_{\alpha, \beta}(p)v \in \ior (\mathcal{C}^u)$ for every $p \in B$ with $f_{\alpha, \beta}(p) \in I \times \mathbb{R} \times \mathbb{R}$ and every $v \in \mathcal{C}^u$.
    \end{enumerate}
\end{deff}

\begin{figure}[h]
    \centering
    \includegraphics[width=1\linewidth]{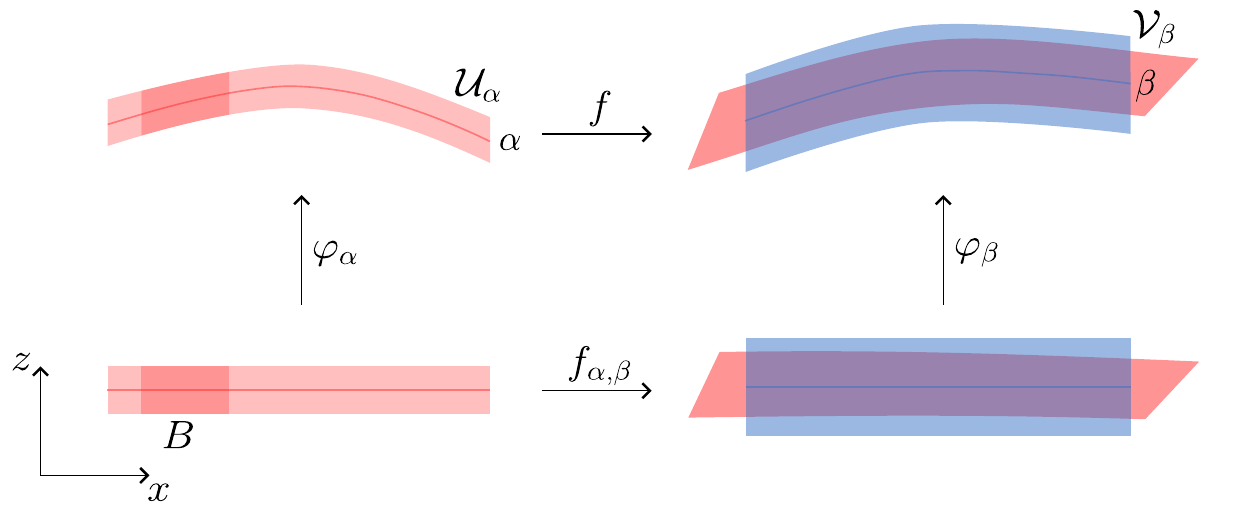}
    \caption{The projection to the $xz$-plane of a through box $B$ for the u-curves $\alpha$ and $\beta$ (\Cref{deff:through_box}).}
    \label{fig:fab}
\end{figure}

\begin{lemma} \label{lemma_verif}
    If $B$ is a through box for $\alpha$ and $\beta$, then $\mathcal{U}_{\alpha}$ maps through $\mathcal{V}_{\beta}$.
\end{lemma}
\begin{proof}

Consider a curve $\gamma$ with $\varphi_{\alpha}(\gamma) \in \mathcal{U}_{\alpha}$. By the definition of $\mathcal{U}_{\alpha}$, every curve in $\mathcal{U}_{\alpha}$ must be of this form. Suppose that $B = T \times \boldsymbol{\epsilon_y} \times \boldsymbol{\epsilon_z}$. From the first item in \Cref{deff:through_box}, it follows that the first coordinate of $f_{\alpha, \beta} ( \gamma|_T)$ contains $I$ in its interior. Thus, there exists $S \subset T$ such that the first coordinate of $f_{\alpha, \beta} ( \gamma|_S)$ is $I$. From the second and third items of \Cref{deff:through_box} and the definition of $\mathcal{V}_{\beta}$ and $\varphi_{\beta}$, it follows that $\varphi_{\beta} \circ f_{\alpha, \beta}(\gamma|_S) \in \ior(\mathcal{V}_{\beta})$. Moreover, since this inclusion is into the interior, this argument holds for every $g$ in a neighbourhood of $f$. Therefore, $\mathcal{U}_{\alpha}$ maps through $\mathcal{V}_{\beta}$.

\end{proof}

This result reduces the problem of verifying that a small u-tube maps through a big u-tube to verifying that an associated box is a through box. The latter can be checked computationally by verifying a finite amount of interval inclusions. The framework up to this point can be generalised to other settings. Next, we describe the computational implementation used to verify that a given box is a through box for $f$. This methodology is tailored specifically to the map $f$ given by \eqref{eq:H}.

Throughout, we consider the interval $I= [-1 + b \epsilon_y, 1 - b \epsilon_y]$ and the box $\mathcal{D} = I \times [-1, 1] \times [-\frac{1}{\xi -1}, \frac{1}{\xi -1}]$, where $\epsilon_y$ corresponds to the size of the u-tubes as described previously. Note that if $\epsilon_y$ is chosen sufficiently small, then the maximal invariant set $\Lambda$ of $\mathcal{D}$ coincides with that of $[-1, 1] \times [-1, 1] \times [-\frac{1}{\xi -1}, \frac{1}{\xi -1}]$. 

We begin by considering two u-curves \[\alpha(t) = (t, p_y(t), p_z(t)) \qandq \beta(t) = (t, q_y(t), q_z(t)).\] 
We will check that a box $B = T \times \boldsymbol{\epsilon_y} \times \boldsymbol{\epsilon_z}$ is a through box for $\alpha$ and $\beta$. Although $\alpha$ and $\beta$ are defined on $I$, for computational purposes we consider them to be extended to the larger interval $[-1, 1]$. We can explicitly write $f_{\alpha, \beta}$ as \[
f_{\alpha,\beta}(x,y,z)
=
\bigl(
\hat{x} + b y,\;
x - q_y(\hat{x} + b y),\;
\xi\bigl(z + p_z(x)\bigr) + x - q_z(\hat{x} + b y)
\bigr),
\] where $\hat{x} = ax^2 + b p_y(x) + c.$ The inputs to this step are the interval $T$ and the coefficients of the polynomials $p_y$, $p_z$, $q_y$ and $q_z$.

In many places in the code, we compute an interval enclosure of a single-variable polynomial $p(x)$ evaluated on an interval $J$. To do this we first compose $p$ with an affine map and reduce to the case of evaluating a polynomial $q(x)$ on $[-1, 1]$. We represent $q(x)$ as a Chebyshev series \[q(x) = \sum_{n=0}^N a_n T_n(x),\] where the $T_n$ are Chebyshev polynomials of the first kind. The first four terms $\sum_{n=0}^3 a_n T_n(x)$ give a cubic polynomial, whose image we can bound tightly by solving for its critical points. The fact that $|T_n(x)| \leq 1$ holds for all $n$ and all $|x| \leq 1$ gives the bound $|\sum_{n=4}^{N} a_n T_n(x)| \leq \sum_{n=4}^N |a_n|$ for the remaining terms.

To verify that $B$ is a through box, we first check that $T\subset I$. If this holds, we proceed to verify the three conditions listed in \Cref{deff:through_box}. To check the first condition, we compute interval enclosures of the images under $f_{\alpha, \beta}$ of the two connected components of $\partial T \times \boldsymbol{\epsilon_y} \times \boldsymbol{\epsilon_z}$. We then verify that one of them lies in $(-\infty, -1 + b \epsilon_y) \times \mathbb{R} \times \mathbb{R}$ and the other in $(1 - b \epsilon_y, \infty) \times \mathbb{R} \times \mathbb{R}$. In fact, we only need to compute the first coordinates of the interval enclosures to verify these inclusions.

To check the second condition in \Cref{deff:through_box}, let $(x, y, z) \in B$ with $f_{\alpha, \beta}(x, y, z) \in I \times \mathbb{R} \times \mathbb{R}$. The second coordinate of $f_{\alpha, \beta}(x, y, z)$ can be bounded as
\[
\bigl| x - q_y(\hat{x} + b y) \bigr|
\;\leq\;
\bigl| x - q_y(\hat{x}) \bigr|
\;+\;
\bigl| q_y(\hat{x}) - q_y(\hat{x} + b y) \bigr| .
\]
By the mean value theorem, $\bigl| q_y(\hat{x}) - q_y(\hat{x} + b y) \bigr| \leq \lVert q_y' \rVert \, \lvert b y \rvert$, where $\lVert q_y' \rVert$ denotes the supremum norm of $q_y'$ over the interval with endpoints $\hat{x}$ and $\hat{x} + by$. Since $f_{\alpha, \beta}(x, y, z) \in I \times \mathbb{R} \times \mathbb{R}$ we have $\hat{x} + by \in I \subset [-1, 1]$. Moreover, since $(x, y, z) \in B$, we have $|y| < \epsilon_y$ and $\hat{x} \in [-1, 1]$. It follows that the interval with endpoints $\hat{x}$ and $\hat{x} + by$ is contained in $[-1, 1]$, therefore the norm of $q_y'$ may be taken over $[-1,1]$. Combining the previous estimate with the bound $|y| < \epsilon_y$, we obtain \[
\bigl| x - q_y(\hat{x} + b y) \bigr|
\;<\;
\bigl| x - q_y(\hat{x}) \bigr|
\;+\;
\lVert q_y' \rVert \, b \, \epsilon_y .
\]
We compute an interval enclosure of this bound over $x \in T$ by evaluating single-variable polynomials on intervals as described above, and verify that it is less than $\epsilon_y$. Recall that $\hat{x}$ is actually a polynomial in $x$ and so $x - q_y(\hat{x})$ is also a polynomial in $x$.

Analogously, the third coordinate of $f_{\alpha, \beta}(x, y, z)$ can be bounded by \[
\bigl| \xi\bigl(z + p_z(x)\bigr) + x - q_z(\hat{x} + b y) \bigr|
\;<\;
\bigl| \xi\, p_z(x) + x - q_z(\hat{x}) \bigr|
\;+\;
\xi\, \epsilon_z
\;+\;
\lVert q_z' \rVert \, b \, \epsilon_y .
\]
We compute an interval enclosure of this bound and verify that it is less than $k \epsilon_z$.

To check the third condition in \Cref{deff:through_box}, we compute the Jacobian \[
D f_{\alpha,\beta}
=
\begin{pmatrix}
\hat{x}' & b & 0 \\[1mm]
1 - q_y'(\hat{x} + b y)\,\hat{x}' & -b\, q_y'(\hat{x} + b y) & 0 \\[1mm]
\xi\, p_z'(x) + 1 - q_z'(\hat{x} + b y)\,\hat{x}' & -b\, q_z'(\hat{x} + b y) & \xi
\end{pmatrix}
\] at a point $(x, y, z) \in B$ with $f_{\alpha, \beta}(x, y, z) \in I \times \mathbb{R} \times \mathbb{R}$. Here $\hat{x}' = \frac{\partial}{\partial x} \hat{x} = 2ax + b p_y'(x)$. We denote by $[Df_{\alpha, \beta}]$ an interval enclosure of $D f_{\alpha,\beta}$ over all such points, obtained by applying the methods described above to each matrix entry. Recall that the unstable cone $\mathcal{C}^u$ consists of all real scalings of vectors in $\{1\} \times \boldsymbol{\epsilon_y} \times \boldsymbol{\hat{\epsilon}_z}$. We let $\boldsymbol{v} = (1, \boldsymbol{\epsilon_y}, \boldsymbol{\hat{\epsilon}_z})$ and compute \[L_x \times L_y \times L_z := [Df_{\alpha, \beta}] (\boldsymbol{v}). \] We verify that $\frac{L_y}{L_x} \subset \ior(\boldsymbol{\epsilon_y})$ and $\frac{L_z}{L_x} \subset \ior(\boldsymbol{\hat{\epsilon}_z})$. Note that rescaling the vector $\boldsymbol{v}$ does not affect these inclusions; therefore, $\boldsymbol{v}$ represents the entire unstable cone $\mathcal{C}^u$.

This concludes the description of \Cref{alg:verif} for the H\'enon-like family of maps.

\section{Construction of u-curves for the H\'enon-like system} \label{sec:constru_spec}

We can use the results from the previous sections to prove \Cref{theo:henon} provided that we have the lists $\mathcal{L}$ and $\mathcal{M}$ (see \Cref{alg:verif}). In this section, we detail how these lists are generated for the H\'enon-like family following \Cref{alg:main}.

Recall that a key step in both \Cref{alg:main} and \Cref{alg:verif} is to check that a small u-tube maps through a big u-tube by verifying that a given box $B = T \times \boldsymbol{\epsilon_y} \times \boldsymbol{\epsilon_z}$ is a through box (see \Cref{deff:through_box}). Since $\epsilon_y$ and $\epsilon_z$ are fixed, the box $B$ is fully determined by the interval $T$, which we compute and store in $\mathcal{M}$ as we generate the lists $\mathcal{L}$ and $\mathcal{M}$. We define $T$ as a minimal interval such that the $x$-coordinate of $f(\alpha(T) + (0, \boldsymbol{\epsilon_y}, \boldsymbol{\epsilon_z}))$ contains the interval $[-1,1]$. This guarantees that the first condition in \Cref{deff:through_box} holds. Since the $x$-coordinate of $f$ is monotone, the lower and upper bounds of $T$ can be computed directly by finding roots of polynomials via Newton's method. Moreover $T$ is chosen to be minimal to ensure tight interval enclosures when verifying the remaining conditions in \Cref{deff:through_box}. 

There may be up to two possible choices for a minimal interval $T$, as shown in \Cref{fig:T}. Making the right choice between them is crucial for the proof to work. We choose by finding the two values of $t$ for which the $x$-coordinate of $f(\alpha(t))$ is 0. Among these two candidates, we compare the corresponding $z$-coordinates of $f(\alpha(t))$ and choose the one that is closer to 0. Moreover, as we initialise \Cref{alg:main} (Step 1) we define the first u-curve as $\beta_1(t) = (t, y_0, 0)$, for some constant $y_0$. Together, these two choices ensure that the absolute values of the $z$-coordinates of our u-curves remain bounded; otherwise, they would grow under iteration, causing the algorithm to fail to terminate.

\begin{figure}[h]
  \centering
  \def\svgwidth{1\linewidth}
  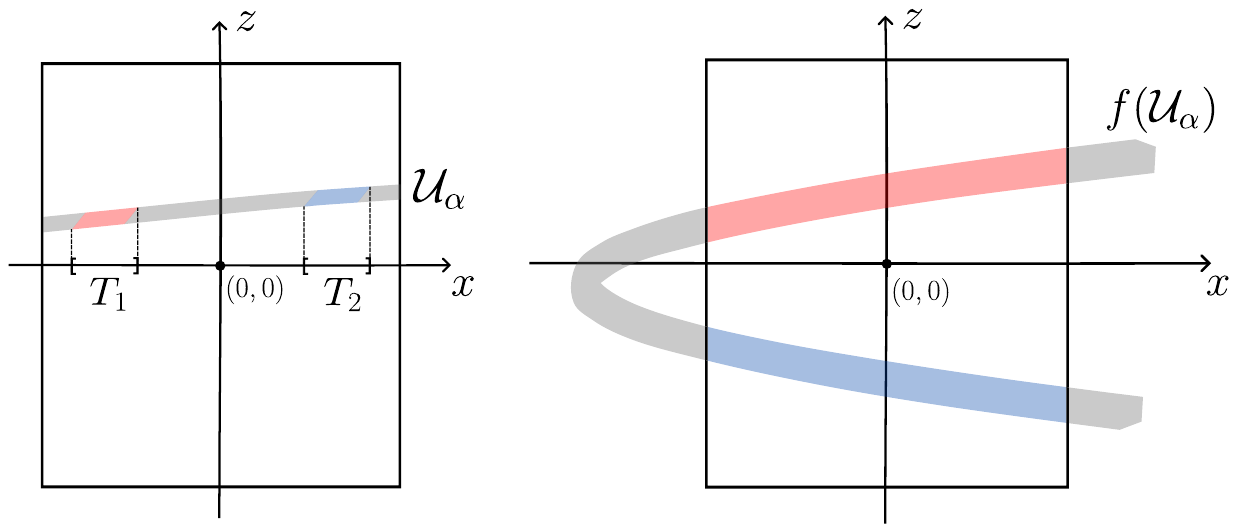
  \caption{The figure shows a small u-tube $\mathcal{U}_\alpha$ (left) and its image $f(\mathcal{U_\alpha})$ (right) both projected to the $xz$-plane. On the right, the red and blue regions represent the intersection of $f(\mathcal{U}_{\alpha})$ with the box $[-1, 1] \times [-1, 1] \times [-\frac{1}{\xi - 1}, \frac{1}{\xi - 1}]$. On the left, we show the corresponding preimage; its projection to the $x$-axis determines the intervals $T_1$ and $T_2$. In this example, we select $T = T_1$, since the intersection point between $f(\alpha|_{T_1})$ and $\{x = 0\}$ is closer to $\{z = 0\}$ than the analogous intersection associated with $T_2$.}
 \label{fig:T}
\end{figure}

Once $T$ is determined, we have all the elements required to verify that a small u-tube maps through a big u-tube. However, for some u-curves $\alpha$ there might not be any existing u-curve $\beta$ for which the small u-tube around $\alpha$ maps through the big u-tube around $\beta$. In that case we construct $\beta$ as an approximation of a subcurve of $f(\alpha)$. We do this by a collocation method. Let $\{x_0, \dots , x_{N} \}$ be a finite set of points in $[-1, 1]$. By Newton's method, we find points $\{t_0, \dots, t_{N}\} \subset T$ such that the $x$-coordinate of $f(\alpha(t_i))$ is equal to $x_i$ for all $i = 0, \dots, N$. We define a u-curve $\beta$ with $\beta(t) = (t, q_y(t), q_z(t))$ by setting $q_y$ and $q_z$ to be the unique polynomials of degree $N$ such that for all $i = 0, \dots, N$, $\beta(x_i) = f \circ \alpha (t_i)$. For numerical stability reasons, we take $\{x_0, \dots , x_{N} \}$ to be Chebyshev nodes, that is $x_i = -\cos(\frac{i\pi}{N})$ for all $i = 0, \dots, N$ \cite{MR1776072}. This collocation procedure is illustrated in \Cref{fig:coll_points}.

\begin{figure}[h!]
  \centering
  \def\svgwidth{1\linewidth}
  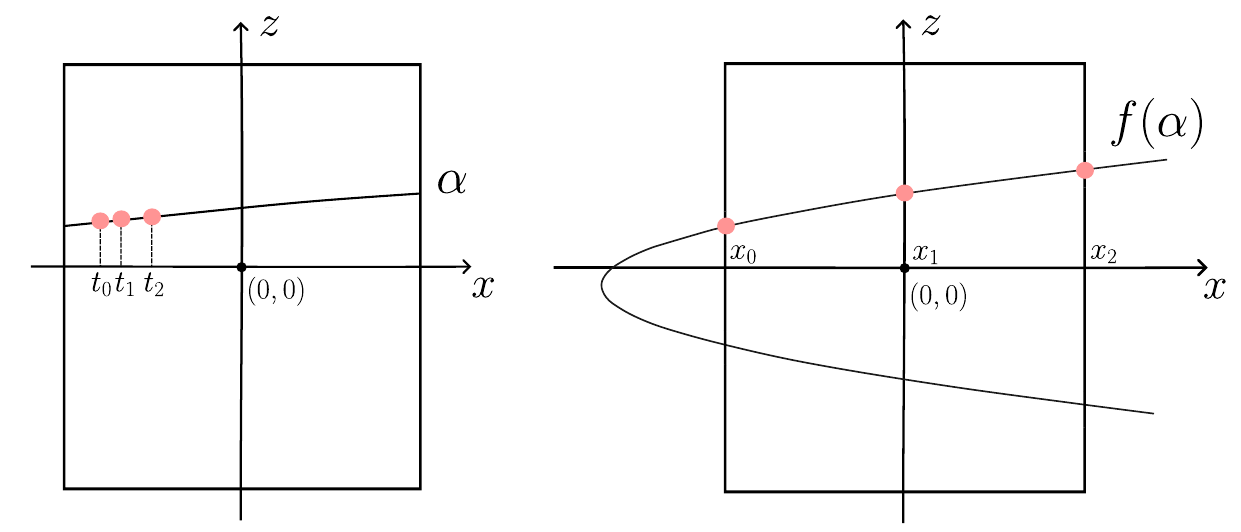
  \caption{ This figure shows a u-curve $\alpha$ (left) and its image $f(\alpha)$ (right) both projected to the $xz$-plane. On the right, in red we represent the intersection points between $f(\alpha|_T)$ and $\{x = x_i\}$, for $i = 0, 1, 2$. On the left, we show the corresponding preimages; their projection to the $x$-axis determine the points $t_0, t_1, t_2 \subset T$.  
 }
 \label{fig:coll_points}
\end{figure}

At this point, all components of \Cref{alg:main} have been specified. However, searching for a u-curve $\beta \in \mathcal{L}$ that a given u-curve $\alpha$ maps through can be costly if done naively. For this reason, we organise the set of u-curves using a hash table. We partition the cross-section $\{x = 0\} \cap \mathcal{D}$ into small rectangular cells and assign each u-curve to the cell where it intersects this plane. The hash table stores, for each cell, the list of u-curves assigned to it. Then, given a u-curve $\alpha$, we first identify which cell is intersected by $f(\alpha(T))$. We then test $\mathcal{U}_\alpha \mapsthrough{H} \mathcal{V}_{\beta}$ only for the u-curves $\beta$ stored in this cell and in neighbouring cells, instead of comparing against the entire collection. This significantly reduces the number of candidates. Additionally, before applying the rigorous verification test, we further discard candidates by checking that the u-curves remain close at $x = x_i$, for $i = 0, \dots, N$.

\section{Results} \label{sec:results}

In this section, we present the results of running \Cref{alg:main} for the H\'enon-like family given by \Cref{eq:H}. The algorithm computes the collection of u-curves needed for the verification, and we report how the choice of computational settings affect the runtime and the number of u-curves generated. 

The outcome of the program depends on several parameters that balance accuracy and computational cost. The parameters $\epsilon_y$ and $\epsilon_z$ determine the size of the u-tubes. Choosing smaller values reduces overestimation, which facilitates the verification of the required conditions and helps prevent u-curves from escaping the box $\mathcal{D}$. On the other hand, excessively small values lead to a larger number of u-tubes, increasing the overall computation time and memory usage. Accuracy is also influenced by the number of collocation points. While a larger number of points improves the quality of the approximation, it also results in higher-degree polynomials, which slow down the computations. Finally, the integers $n$ and $k$ determine the opening of the unstable cones $\mathcal{C}^u$. Wider cones correspond to larger values of $n$ and smaller values of $k$, making the verification conditions easier to satisfy, but at the cost of increased computational cost due to the larger number of u-tubes generated. 

To ensure the result holds for a continuous range of values of $\xi$, rather than running the algorithm at single points, we proceed over intervals as follows. For a given $\xi$-interval, we use its upper bound to determine the domain $\mathcal{D}$. Since the $z$-coordinate of $\mathcal{D}$ increases as $\xi$ decreases, this choice is applicable to every $\xi$ in the interval. To construct the family of u-curves, we fix $\xi$ at the midpoint of the interval and execute \Cref{alg:main} as detailed in \Cref{sec:constru_spec}. Then, we verify the blender condition by executing \Cref{alg:verif} over the entire $\xi$-interval, as detailed in \Cref{sec:verif_spec}. 

For different intervals of $\xi$, we use different settings that balance accuracy and computational cost. In \Cref{table:xi_1}, we summarise the computational settings used for small values of $\xi$, along with the number of u-curves in the family at the time of termination and the total runtime. We note that for $\xi = 1$, even though we construct a family of curves, there is no blender since $\Lambda$ is not a hyperbolic set.

\begin{table}[h]
\centering
\setlength{\tabcolsep}{8pt}
\begin{tabular}{|c|c|c|c|c|c|c|c|}
\hline
$\xi$ & $\epsilon_y$ & $\epsilon_z$ & \makecell{Number of \\ coll. points} & $n$ & $k$ & \makecell{Number \\ of u-curves} & \makecell{Total \\ runtime} \\ 
\hline
$[1, 1.1]$ & $0.02$ & $1.1$ & 3 & 3 & 2 & 8 & 0.009 s\\ 
\hline
$[1.1, 1.2]$ & $0.02$ & $0.8$ & 3 & 3 & 2 & 9 & 0.009 s\\
\hline
$[1.2, 1.3]$ & $0.02$ & $0.5$ & 3 & 3 & 2 & 15 & 0.010 s\\
\hline
\end{tabular}
\caption{Depending on the range of $\xi$, we report the necessary settings for \Cref{alg:main} to terminate. All runtimes were measured on a machine with an Intel Core i5 processor and 16 GB of RAM.}
\label{table:xi_1}
\end{table}
As $\xi$ increases beyond 1.3, verification becomes more challenging, and the settings require more careful tuning, as summarised in \Cref{table:xi_2}. In these cases, the verification is carried out by subdividing the corresponding $\xi$-interval into smaller subintervals, as indicated in the second-last column of \Cref{table:xi_2}. Larger values of $\xi$ require smaller subintervals for the algorithm to succeed.

\begin{table}[h]
\centering
\setlength{\tabcolsep}{3.5pt}
\begin{tabular}{|c|c|c|c|c|c|c|c|c|}
\hline
$\xi$ & $\epsilon_y$ & $\epsilon_z$ & \makecell{Number of \\ coll. points} & $n$ & $k$ & \makecell{Max. number \\ of u-curves} & \makecell{Number of \\ subintervals} & \makecell{Total \\ runtime} \\ 
\hline
$[1.3, 1.4]$ & $0.02$ & $0.2$ & 5 & 3 & 2 & $34$& 10 & $0.231$ s \\
\hline
$[1.4, 1.5]$ & $0.02$ & $0.1$ & 5 & 3 & 2 & $119$& 10 & $0.375$ s \\
\hline
$[1.5, 1.6]$ & $0.02$ & $0.1$ & 5 & 5 & 2 & $263$& 10 & $0.770$ s \\
\hline
$[1.6, 1.65]$ & $0.01$ & $0.05$ & 5 & 5 & 2 & $928$ & 5 & $1.648$ s \\
\hline
$[1.65, 1.7]$ & $0.01$ & $0.01$ & 5 & 5 & 2 & $16799$ & 50 & $6.7$ min \\
\hline
$[1.7, 1.712]$ & $0.004$ & $0.004$ & 5 & 5 & 2 & $140483$& 12 & $31.5$ min \\
\hline
\end{tabular}
\caption{Depending on the range of $\xi$, we report the necessary settings for \Cref{alg:main} to terminate. The values given for the number of u-curves correspond to the maximum observed among all subintervals within the corresponding $\xi$ range. All runtimes were measured on a machine with an Intel Core i5 processor and 16 GB of RAM.}
\label{table:xi_2}
\end{table}
Running the program using the settings specified in \Cref{table:xi_1} and \Cref{table:xi_2} proves \Cref{theo:henon}. For larger values of $\xi$, similar settings still lead to successful runs, but only on very small $\xi$-intervals, as reported in \Cref{table:xi_3}. Covering a larger interval would require multiple runs of the algorithm and the storage of a rapidly growing number of u-curves, which quickly leads to memory limitations. In particular, for $\xi = 1.716$, the program could only be run for a single parameter value.
\begin{table}[h]
\centering
\setlength{\tabcolsep}{5.5pt}
\begin{tabular}{|c|c|c|c|c|c|c|c|}
\hline
$\xi$ & $\epsilon_y$ & $\epsilon_z$ & \makecell{Number of \\ coll. points} & $n$ & $k$ & \makecell{Number \\ of u-curves} & \makecell{Total \\ Runtime} \\
\hline
$[1.71299, 1.71301]$ & $0.001$ & $0.002$ & 7 & 5 & 2 & $118727$ & $2.6$ min\\
\hline
$[1.7139999, 1.7140001]$ & $0.001$ & $0.002$ & 7 & 5 & 2 & $128715$ & $3.0$ min\\
\hline
$[1.7149999, 1.7150001]$ & $0.001$ & $0.0015$ & 7 & 5 & 2 & $205745$ & $5.0$ min\\
\hline
$[1.716, 1.716]$ & $0.001$ & $0.0005$ & 7 & 5 & 2 & $2460876$ & $119.7$ min\\
\hline
\end{tabular}
\caption{ Depending on the range of $\xi$, we report the necessary settings for \Cref{alg:main} to terminate. All runtimes were measured on a machine with an Intel Core i5 processor and 16 GB of RAM.
}
\label{table:xi_3}
\end{table}
\section{Discussion} \label{sec:discussion}

We have presented an algorithm to detect blenders by building a family of curves from a single curve that is roughly aligned with the strong unstable direction. As an application, we prove the existence of blenders for a H\'enon-like family of maps. This approach extends the range of parameters for which this family of maps is known to support a blender. 

A natural question in the study of this family is to determine the threshold of $\xi$ beyond which blenders are lost. The work of \cite{MR4160092} suggests that blenders persist up to some value $\xi_0$ at roughly around $1.75$, as illustrated in Figure 8 of \cite{MR4160092}. Although we have rigorously confirmed that blenders exist for $\xi \in (1, 1.712]$, for higher values of $\xi$ our program can only be run on very small $\xi$-intervals. As a result, even though we verified the existence of blenders for some individual values up to $\xi = 1.716$ (see \Cref{table:xi_3}), memory limitations prevent us from covering the entire range. Nonetheless, we believe that our algorithm would succeed for any individual value of $\xi$ up to $1.716$. Improving the way our program stores u-curve data could potentially allow verification for higher values of $\xi$.

Although a complete characterisation of this threshold lies beyond the scope of this paper, we provide some numerical evidence to shed light on this question. Fixing values of $\epsilon_y$, $\epsilon_z$, $n$ and $k$ that yield a successful result for $\xi = 1.716$, we ran our program for values of $\xi$ between $1.1$ and $1.716$ with uniform spacing $0.01$. In \Cref{fig:number_of_curves} we plot the number of u-curves generated for each value of $\xi$. We observe that this number grows more rapidly as $\xi$ increases, though it remains unclear whether it approaches an asymptote. In \cite{CKO2025}, the authors compute large pieces of a stable manifold and its intersection points with a fixed section, project these intersection points to a line, and study the gaps between them as $\xi$ varies. The absence of gaps suggests the presence of a blender. Although the parameter values for the first two coordinates in \cite{CKO2025} differ from ours, D. C'Julio (personal communication, July 2026) applied this procedure, combined with the method introduced in \cite{CADKO26}, to our choice of values $(a,b,c)$ using a fixed section aligned to the unstable direction and found that the largest gap appears at $\xi^* \approx 1.728546$. This is consistent with our findings: there are no gaps for the values of $\xi$ for which we prove the existence of blenders. As $\xi$ approaches $\xi^*$, the computational time and memory required by our program grow rapidly, making verification intractable before reaching $\xi^*$.

\begin{figure}[h]
    \centering
    \includegraphics[width=1\linewidth]{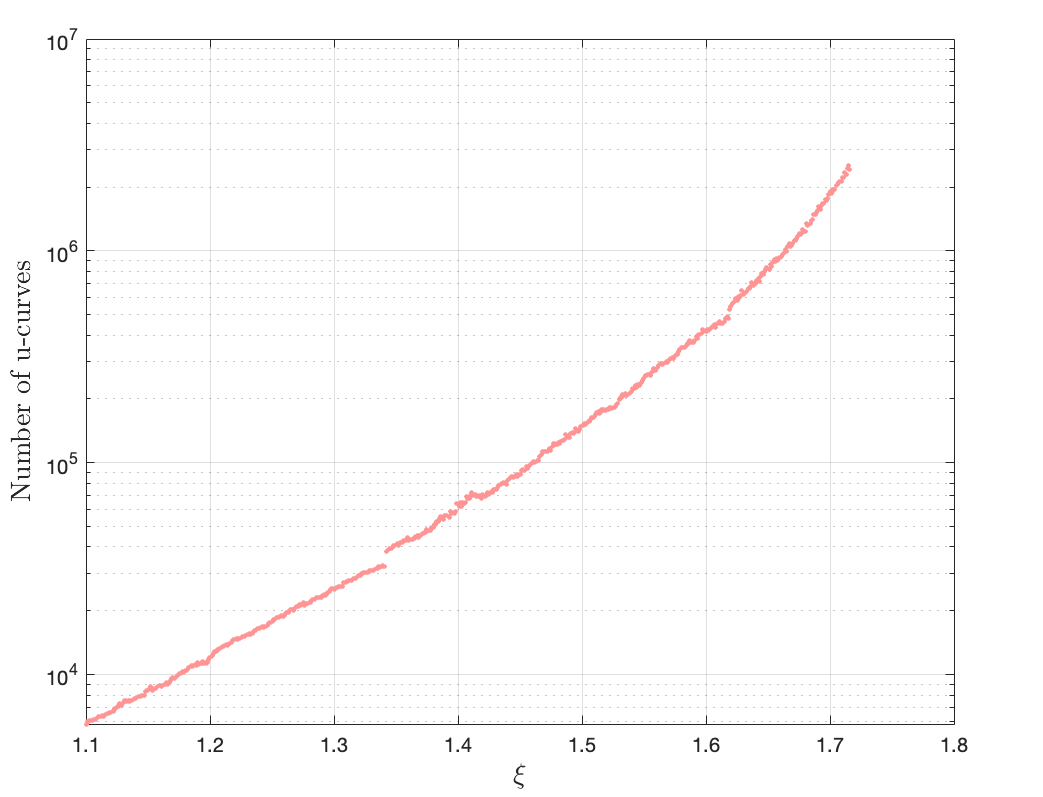}
    \caption{Number of u-curves generated for fixed values of $\epsilon_y$, $\epsilon_z$, $n$ and $k$ as a function of $\xi$, in logarithmic scale, for values of $\xi$ between 1.1 and 1.716 with uniform spacing 0.01. The u-curves are constructed non-rigorously for computational efficiency. Nonetheless, the numbers of u-curves agree with the results of \Cref{sec:results} at selected values of $\xi$, confirming that this plot provides a reliable qualitative picture. }
    \label{fig:number_of_curves}
\end{figure}

All of the above discussion concerns maps with the parameter $b = \beta = 0.3$. We also ran our program for the same family of maps with $b = 0.1$, as studied in \cite{MR3846435}, and observed similar results. We established the existence of blenders for $\xi \in (1, 1.712]$, and successfully verified the result for individual values up to $\xi = 1.719$, slightly beyond what we achieved for $b = 0.3$.

\FloatBarrier

\bibliographystyle{alpha}
\bibliography{sample}

\end{document}